\documentclass[11pt,reqno]{article}

\usepackage{amsmath,amsthm,amsfonts,amssymb,latexsym,mathrsfs,color}

\newtheorem{theorem}{Theorem}

\newtheorem{corollary}[theorem]{Corollary}
\newtheorem{proposition}[theorem]{Proposition}

\newcommand\mm{{\rm mod\ }}

\newcommand{\des}{{\rm des\,}}

\newcommand{\msn}{{\mathcal S}_n}

\newcommand{\sn}{S_n}
\newcommand{\lrf}[1]{\lfloor #1\rfloor}

\title{An explicit formula for the number of permutations with a given number of alternating runs \footnote{This work is supported by the Fundamental Research Funds for the Central Universities (N100323013).}}

\author
{Shi-Mei Ma \footnote{ {\it Email address:}
shimeima@yahoo.com.cn (S.-M. Ma)} }
\date{\footnotesize Department of Information and Computing Science,
        Northeastern University at Qinhuangdao,\\ Hebei 066004,
        China}

\begin{document}

\maketitle

\begin{abstract}
Let $R(n,k)$ denote the number of
permutations of $\{1,2,\ldots,n\}$ with $k$ alternating runs.
In this note we present an explicit formula for the numbers $R(n,k)$.
\end{abstract}
\section{Introduction}
Let $\msn$ denote the symmetric group of all permutations of $[n]$, where $[n]=\{1,2,\ldots,n\}$.
Let $\pi=\pi(1)\pi(2)\cdots\pi(n)\in\msn$.
We say that $\pi$ changes
direction at position $i$ if either $\pi({i-1})<\pi(i)>\pi(i+1)$, or
$\pi(i-1)>\pi(i)<\pi(i+1)$, where $i\in\{2,3,\ldots,n-1\}$. We say that $\pi$ has $k$ {\it alternating
runs} if there are $k-1$ indices $i$ such that $\pi$ changes
direction at these positions. Let $R(n,k)$ denote the number of
permutations in $\sn$ with $k$ alternating runs. Andr\'e~\cite{Andre84} was the first to study
the alternating runs of permutations and he obtained the following recurrence
\begin{equation}\label{rnk-rr}
    R(n,k)=kR(n-1,k)+2R(n-1,k-1)+(n-k)R(n-1,k-2)
\end{equation}
for $n,k\ge 1$, where $R(1,0)=1$ and $R(1,k)=0$ for $k\ge 1$.

Let
$R_n(x)=\sum_{k=1}^{n-1}R(n,k)x^k$.
Then the recurrence (\ref{rnk-rr}) induces
\begin{equation*}
R_{n+2}(x)=x(nx+2)R_{n+1}(x)+x\left(1-x^2\right)R_{n+1}'(x),
\end{equation*}
with initial values $R_1(x)=1$, $R_2(x)=2x$ and $R_3(x)=2x+4x^2$.
For $2\leq n\leq 5$, the coefficients of $R_{n}(x)$
can be arranged as follows with $R({n,k})$ in row $n$ and column $k$:
$$\begin{array}{ccccccc}
  2 & &  & & &\\
  2 & 4 &  & &  &\\
  2 & 12 & 10 &  & &\\
  2& 28 & 58 & 32&  &
\end{array}$$

For a permutation $\pi=\pi(1)\pi(2)\cdots \pi(n)\in\msn$, we define a {\it descent} to be a position $i$ such that $\pi(i)>\pi(i+1)$. Denote by $\des(\pi)$ the number of descents of $\pi$. The generating function for descents is
\begin{equation*}
A_n(x)=\sum_{\pi\in\msn}x^{1+\des(\pi)},
\end{equation*}
which is well known as the {\it Eulerian polynomial}.
The polynomial $R_n(x)$ is
closely related to the Eulerian polynomial $A_n(x)$:
\begin{equation}\label{rnx-anx}
R_n(x)=\left(\dfrac{1+x}{2}\right)^{n-1}(1+w)^{n+1}A_n\left(\dfrac{1-w}{1+w}\right),\quad
w=\sqrt{\frac{1-x}{1+x}},
\end{equation}
which was first established by David and
Barton~\cite[157-162]{DB62} and then stated more concisely by
Knuth~\cite[p.~605]{Knuth73}.
In a series of papers~\cite{Carlitz78,Carlitz80,Carlitz81}, Carlitz studied the
generating functions for the numbers $R(n,k)$. In particular,
Carlitz~\cite{Carlitz78} obtained the generating function
\begin{equation}\label{CarlitzGF}
\sum_{n=0}^\infty (1-x^2)^{-n/2}\frac{z^n}{n!}\sum_{k=0}^nR(n+1,k)x^{n-k}=\frac{1-x}{1+x}\left(\frac{\sqrt{1-x^2}+\sin z}{x- \cos z}\right)^2.
\end{equation}
In~\cite{BE00}, B\'ona and Ehrenborg proved that $R_n(x)$ has the zero $x=-1$ with
multiplicity $\lrf{\frac{n}{2}}-1$.

Recently there has been much interest in obtaining explicit formula for the numbers $R(n,k)$.
For example, Stanley~\cite{Sta08} proved that
\begin{equation*}\label{Rnk-Stan}
R(n,k)=\sum_{i=0}^k\frac{1}{2^{i-1}}(-1)^{k-i}z_{k-i}\sum_{\stackrel{r+2m\leq i}{r\equiv i~\mm 2}}(-2)^m\binom{i-m}{(i+r)/2}\binom{n}{m}r^n,
\end{equation*}
where $z_0=2$ and all other $z_i$'s are 4.
In~\cite{CW08}, Canfield and Wilf showed that
$$R(n,k)=\frac{1}{2^{k-2}}k^n-\frac{1}{2^{k-4}}(k-1)^n+\psi_2(n,k)(k-2)^n+\cdots+\psi_{k-1}(n,k)\quad\textrm{for $k\ge 2$},$$
in which each $\psi_i(n,k)$ is a polynomial in $n$ whose degree in $n$ is $\lrf{i/2}$.
Moreover, Canfield and Wilf~\cite{CW08} also noted that Carlitz~\cite{Carlitz80} stated an incorrect explicit formula for the numbers $R(n,k)$.

As an extension of the work of Canfield and Wilf~\cite{CW08},
by using derivative polynomials introduced by Hoffman~\cite{Hoffman95}, in Section 3 we find an explicit formula for the numbers $R(n,k)$.
\section{Derivative polynomials}
Let $D$ denote the differential operator $\frac{d}{d\theta}$. Set $x= \tan \theta$. Then $D(x)=1+x^2$ and $D(x^n)=nx^{n-1}(1+x^2)$ for $n\geq 1$. Thus $D^n(x)$ is a polynomial in $x$. Let
$P_n(x)=D^n(x)$.
Then $P_0(x)=x$ and $P_{n+1}(x)=(1+x^2)P_n'(x)$. Clearly, $\deg P_n(x)=n+1$.
The first few terms can be computed directly as follows:
\begin{eqnarray*}
  P_1(x) &=& 1+x^2, \\
  P_2(x)&=& 2x+2x^3, \\
  P_3(x)&=& 2+8x^2+6x^4,\\
  P_4(x)&=& 16x+40x^3+24x^5.
\end{eqnarray*}
Let $P_n(x)=\sum_{k=0}^{n+1}p(n,k)x^k$. It is easy to verify that
\begin{equation*}
p(n,k)=(k+1)p(n-1,k+1)+(k-1)p(n-1,k-1).
\end{equation*}
Note that $P_n(-x)=(-1)^{n+1}P_n(x)$.
Hence we have the following expression:
\begin{equation}\label{pnk}
P_n(x)=\sum_{k=0}^{\lrf{\frac{n+1}{2}}}p(n,n-2k+1)x^{n-2k+1}.
\end{equation}

The study of properties of the polynomials $P_n(x)$, initiated in~\cite{Hoffman95}, is presently very active.
We refer the reader to~\cite{Cvijovic09,Cvijovic101,Franssens07,Hoffman99} for recent progress on this subject.
For example, let
$$\tan^k (t)=\sum_{n=k}^\infty T(n,k)\frac{t^n}{n!}.$$
The numbers $T(n,k)$ are  called {\it the tangent numbers of order $k$} (see~\cite[p.~428]{Carlitz72}).
Cvijovi\'c~\cite[Theorem 1]{Cvijovic09} obtained the following formula
$$P_n(x)=T(n,1)+\sum_{k=1}^{n+1}\frac{1}{k}T(n+1,k)x^k.$$

In the following discussion, we will present an explicit expression for the numbers $p(n,k)$.
Let $A(n,k)=|\{\pi\in\msn:\des(\pi)=k-1\}|$.
The Eulerian polynomials $A_n(x)$ are defined by
$A_n(x)=\sum_{k=1}^nA(n,k)x^k$. Let $S(n,k)$ be the Stirling number of the second kind. Explicitly,
$$S(n,k)=\frac{1}{k!}\sum_{r=0}^k(-1)^{k-r}\binom{k}{r}r^n.$$
The Eulerian polynomial $A_n(x)$ admits several expansions in terms of different polynomial bases.
One representative example is the Frobenius formula
\begin{equation}\label{Frobenius}
A_n(x)=\sum_{i=1}^ni!S(n,i)x^i(1-x)^{n-i}
\end{equation}
(see~\cite[Theorem 14.4]{Charalambos02} for details).
Setting $x=(y-1)/(y+1)$ in~(\ref{Frobenius}) and then multiplying both sides by $(y+1)^{n+1}$, we get
\begin{equation*}
\sum_{k=1}^nA(n,k)(y-1)^k(y+1)^{n-k+1}=(y+1)\sum_{i=1}^ni!S(n,i)2^{n-i}(y-1)^i.
\end{equation*}
Put
\begin{equation}\label{Any-Fro}
a_n(y)=(y+1)\sum_{i=1}^ni!S(n,i)2^{n-i}(y-1)^i.
\end{equation}
The first few terms of $a_n(y)$'s are given as follows:
\begin{eqnarray*}
  a_1(y) &=& -1+y^2, \\
  a_2(y)&=& -2y+2y^3, \\
  a_3(y)&=& 2-8y^2+6y^4,\\
  a_4(y)&=& 16y-40y^3+24y^5.
\end{eqnarray*}
Then from~\cite[Theorem~5]{Franssens07},
we have
\begin{equation}\label{any}
a_n(y)=\sum_{k=0}^{\lrf{\frac{n+1}{2}}}(-1)^kp(n,n-2k+1)y^{n-2k+1}.
\end{equation}
Therefore, combining~(\ref{Any-Fro}) and~(\ref{any}), we get the following result.
\begin{proposition}\label{BnyPny}
For $n\geq 1$ and $0\leq k\leq {\lrf{\frac{n+1}{2}}}$, we have
\begin{equation*}
p(n,n-2k+1)=\sum_{i\geq 1}i!S(n,i)2^{n-i}(-1)^{i-n+k}\left[\binom{i}{n-2k}-\binom{i}{n-2k+1}\right].
\end{equation*}
\end{proposition}
\section{Explicit Formulas}
In this section, we will obtain an explicit formula for the numbers $R(n,k)$.

Setting $x=\cos 2\theta$ and replacing $z$ by $2z$ in~(\ref{CarlitzGF}), we get
\begin{equation*}
 \sum_{n=0}^\infty \frac{(\sin 2\theta)^{-n}2^nz^n}{n!}\sum_{k=0}^nR(n+1,k){\cos^{n-k} 2\theta}=
   \tan^2 \theta \cot^2(\theta-z).
\end{equation*}
Thus by replacing $z$ by $-z$, we obtain
\begin{equation}\label{Carlitz}
\sum_{n=0}^\infty (-1)^n \frac{(\sin 2\theta)^{-n}2^nz^n}{n!}\sum_{k=0}^nR(n+1,k){\cos^{n-k} 2\theta}=\tan^2 \theta \cot^2(\theta+z).
\end{equation}
By Taylor's theorem, we have
\begin{equation*}\label{cot}
\cot^2(\theta+z)=\sum_{n=0}^\infty D^n(\cot^2 \theta)\frac{z^n}{n!}.
\end{equation*}
Let $y=\cot\theta$. Then $D(y)=-1-y^2$ and $D(y^n)=-ny^{n-1}(1+y^2)$. Put $C_n(y)=D^n(y)$.
Then $C_{n+1}(y)=-(1+y^2)C_n'(y)$ for $n\geq 0$.
Clearly, $C_n(y)=(-1)^nP_n(y)$.
Note that $$D^2(y)=D(D(y))=D(-1-y^2)=-D(y^2).$$
Then
\begin{equation*}\label{DnzPnz}
D^n(y^2)=D^{n-1}(D(y^2))=-D^{n+1}(y)=-C_{n+1}(y)=(-1)^{n}P_{n+1}(y).
\end{equation*}
Therefore, we have
\begin{equation}\label{cotPnz}
\cot^2(\theta+z)=\sum_{n=0}^\infty (-1)^{n}P_{n+1}(\cot\theta)\frac{z^n}{n!}.
\end{equation}

Now we present the main result of this note.
\begin{theorem}\label{mthm}
For $n\geq 2$, we have
\begin{equation}\label{RnxPnx}
R_n(x)=\left(\frac{x+1}{2}\right)^{n-1}\left(\frac{x-1}{x+1}\right)^{\frac{n+1}{2}}P_n\left(\sqrt{\frac{x+1}{x-1}}\right).
\end{equation}
\end{theorem}
\begin{proof}
Substituting~(\ref{cotPnz}) into~(\ref{Carlitz}), we get
\begin{equation*}
 \sum_{n=0}^\infty (-1)^n \frac{(\sin 2\theta)^{-n}2^nz^n}{n!}\sum_{k=0}^nR(n+1,k){\cos^{n-k} 2\theta}
 =\tan^2 \theta \sum_{n=0}^\infty (-1)^{n}\frac{z^n}{n!}\sum_{k=0}^{n+2}p(n+1,k)\cot^k\theta.
\end{equation*}
Equating the coefficients of $(-1)^n{z^n}/{n!}$, we obtain
\begin{equation}\label{exp-1}
(\sin 2\theta)^{-n}2^n\sum_{k=0}^nR(n+1,k){\cos^{n-k} 2\theta}=\tan^2 \theta\sum_{k=0}^{n+2}p(n+1,k)\cot^k\theta.
\end{equation}
Replacing $n$ by $n-1$ in~(\ref{exp-1}), it follows that
\begin{equation*}
 \sum_{k=0}^{n-1}R(n,k){\cos^{n-k-1} 2\theta}
 =\frac{(1+ \cos 2\theta)^{n-1}}{2^{n-1}} \left(\frac{1- \cos 2\theta}{1+ \cos 2\theta}\right)^{\frac{n+1}{2}}\sum_{k=0}^{n+1}p(n,k)\left(\frac{1+ \cos 2\theta}{1- \cos 2\theta}\right)^{\frac{k}{2}}\\
\end{equation*}
Consequently, replacing $\cos 2\theta$ by $x$, we get
\begin{equation}\label{Rnk}
\sum_{k=0}^{n-1}R(n,k)x^{n-k-1}=\left(\frac{1+x}{2}\right)^{n-1}\left(\frac{1-x}{1+x}\right)^{\frac{n+1}{2}}
P_n\left(\sqrt{\frac{1+x}{1-x}}\right).
\end{equation}
Replacing $x$ by $1/x$ and then multiplying both sides of~(\ref{Rnk}) by $x^{n-1}$, the desired result follows.
\end{proof}

Combining~(\ref{pnk}) and~(\ref{RnxPnx}), we get
\begin{equation}\label{R2nx}
R_{n}(x)=\frac{1}{2^{n-1}}\sum_{k=0}^{\lrf{\frac{n+1}{2}}} p(n,n-2k+1)(x+1)^{n-k-1}(x-1)^k.
\end{equation}
Therefore, by~(\ref{R2nx}), we immediately get that
$R_n(x)$ is divisible by $(x+1)^{\lrf{n/2}-1}$.
Denote by $E(n,k,s)$ the coefficients of $x^s$ in $(x+1)^{n-k-1}(x-1)^k$.
It is easy to verify that
$$E(n,k,s)=\sum_{j=0}^{\min (k,s)}(-1)^{k-j}\binom{n-k-1}{s-j}\binom{k}{j}$$
Consequently, by~(\ref{R2nx}), we obtain the following corollary.
\begin{corollary}
For $n\geq 1$ and $1\leq s\leq n-1$, we have
\begin{equation}\label{explicit}
R(n,s)=\frac{1}{2^{n-1}}\sum_{k=0}^{\lrf{\frac{n+1}{2}}}p(n,n-2k+1)E(n,k,s).
\end{equation}
\end{corollary}




\end{document}